\documentclass[12pt]{amsart}
\usepackage{mathtools} % basic maths packages
\usepackage{dsfont} % for mathematical fonts like mathbb{R}
\usepackage{amssymb}

\usepackage{graphicx} % for including pictures with includegraphics
\usepackage{color}
\usepackage[colorlinks,citecolor=blue,linkcolor=blue,urlcolor=blue]{hyperref} % for references in the text with colors
\usepackage{cleveref}

\usepackage[utf8]{inputenc} % to translate accents to intern latex code
\usepackage[english]{babel} % for english language

\usepackage{geometry} % for size, margins, ...
\usepackage{tikz} % for pictures with tikz
\usetikzlibrary{calc,intersections}
\usetikzlibrary{hobby}
\usepackage{cite} %bibtex
\usepackage{float} %positioning figures

\newtheorem{theorem}{Theorem}

\newtheorem{lemma}[theorem]{Lemma}
\newtheorem{coro}[theorem]{Corollary}
\newtheorem{definition}[theorem]{Definition}

\newtheorem{proposition}[theorem]{Proposition}
\newtheorem{example}[theorem]{Example}

\newtheorem{Remark}[theorem]{Remark}

\newtheorem{fakethm}{Theorem}

\numberwithin{theorem}{section} % theorem labeling within sections
\numberwithin{equation}{section} % equation labeling within sections
\numberwithin{figure}{section} % figure labeling within sections

\newcommand*{\R}{\mathbb{R}}
\newcommand*{\C}{\mathbb{C}}

\newcommand*{\Q}{\mathbb{Q}}

\newcommand*{\Z}{\mathbb{Z}}
\newcommand*{\N}{\mathbb{N}}

\DeclareMathOperator{\SL}{\mathrm{SL}}
\DeclareMathOperator{\PGL}{\mathrm{PGL}}

\DeclareMathOperator{\PSL}{\mathrm{PSL}}

\DeclareMathOperator{\tr}{tr}

\title{Nuancing the uniqueness of $q$-rationals}

\author{Perrine Jouteur}
\address{Max Planck Institute for Mathematics in the Sciences, Inselstr. 22, 04103 Leipzig, Germany}
\email{perrine.jouteur@mis.mpg.de}

\author{Olga Paris-Romaskevich}
\address{CNRS, ICJ UMR 5208, École Centrale de Lyon, INSA Lyon, Université Claude Bernard Lyon 1, Université Jean Monnet, 69622 Villeurbanne, France}
\email{paro@math.univ-lyon1.fr}

\author{Alexander Thomas}
\address{CNRS, ICJ UMR 5208, École Centrale de Lyon, INSA Lyon, Université Claude Bernard Lyon 1, Université Jean Monnet, 69622 Villeurbanne, France}
\email{athomas@math.univ-lyon1.fr}

\begin{document}

\begin{abstract}
We prove uniqueness of $q$-rational numbers up to conjugacy, using character varieties. Despite the uniqueness, we exhibit a two-parameter family of deformations of rationals with a modular symmetry. We prove that there are exactly two deformations which deliver the usual $q$-integers: the original $q$-rationals defined by Morier-Genoud and Ovsienko, and another new one. Although the new family can be obtained by conjugacy from the old one, new positivity properties appear. In addition, this new family provides a direct computation of the Jones polynomial of rational knots.
\end{abstract}

\maketitle

\setcounter{tocdepth}{1}
\tableofcontents

%%%%%%%%%%%%%%%%%%%%%%%%%%%%%%%%%%%%%%%%%%%%%%%%%%%%%%%%%%%%%%%%%%%%%%%%%%%
\section{Introduction}
%%%%%%%%%%%%%%%%%%%%%%%%%%%%%%%%%%%%%%%%%%%%%%%%%%%%%%%%%%%%%%%%%%%%%%%%%%%

For a formal parameter $q$, and an integer $n\in\Z$, we define the $q$-integer as 
\begin{equation}\label{Eq:def-n-q}
[n]_q:=\frac{1-q^n}{1-q} \, .
\end{equation}
Integers are recovered as the limit of the right-hand side when $q$ goes to 1. 
The $q$-integers  make their appearance in combinatorics, notably in counting problems in vector spaces over finite fields, but also in representation theory of quantum groups, as well as in knot theory.

In \cite{MGO-2020}, Morier-Genoud and Ovsienko extended this classical $q$-deformation of integers to all rational numbers. Their construction relies on a deformation of the standard action of the modular group $\PSL_2(\Z)$ on the projective line $\mathbb{QP}^1$ by fractional linear transformations. It is defined via the following deformation of the two generators of the modular group $\PSL_2(\Z)= \left<T,S ~|\; S^2=(TS)^3=1\right>$:
\begin{equation}\label{eq:q_representation}
\begin{array}{c c c}
T = \begin{pmatrix}
		1 & 1\\
		0 & 1\\
		\end{pmatrix} & \overset{\varrho_q}{\longmapsto} & 
T_q := \begin{pmatrix}
		q & 1\\
		0 & 1\\
		\end{pmatrix},\\[20pt]	
S = \begin{pmatrix}
		0 & -1\\
		1 & 0\\
		\end{pmatrix} & \overset{\varrho_q}{\longmapsto} &
	 S_q := \begin{pmatrix}
		0 & -1\\
		q & 0\\
		\end{pmatrix}.\\
	\end{array}
\end{equation}

At this formal level, the matrices $T_q$ and $S_q$ are elements of $\PGL_2(\Z[q,q^{-1}])$. One checks that as Möbius transformations, we still have $S_q^2=(T_qS_q)^3 = 1$. Hence the group generated by $T_q$ and $S_q$ is still $\PSL_2(\Z)$ ($q$ is formal here). We denote the representation given in Equation \eqref{eq:q_representation} by $\varrho_q$.

Given a matrix $M\in \PSL_2(\Z)$, its $q$-analogue $M_q$ is obtained by replacing $T$ by $T_q$ and $S$ by $S_q$ in the expansion of $M$ as a finite product of generators $T$ and $S$. The result does not depend on the choice of representing $M$ as product of $T$ and $S$.

\begin{definition}[\cite{MGO-2020, BBL}]\label{def:q-numbers}
The \emph{right} and \emph{left} {$q$-versions} of a number $x \in \mathbb{QP}^1$ are rational functions in $q$ with integer coefficients, denoted $[x]_q^{\sharp}, [x]_q^{\flat}$ respectively. They are defined as
\begin{equation}\label{defi:right_q_numbers}
\left[x\right]^{\sharp}_q := M_q \cdot \frac{1}{0} \;\;\;\;\;\; \text{and} \;\;\;\;\;\;
\left[x\right]^{\flat}_q := M_q \cdot \frac{1}{1-q},
\end{equation}
\noindent where $M\in \PSL_2(\Z)$ is any map such that $M\cdot \frac{1}{0} = x$ with $\PSL_2(\Z)$ acting by M\"obius transformations.
\end{definition}

Each of these two deformations of rationals is defined as the orbit by the $q$-deformed group of one of the two fixed points of $T_q$ which makes the $q$-numbers well-defined. When $q=1$, the two deformations merge into one: $[x]_1^{\sharp}=[x]_1^{\flat}=x$ for all $x \in \mathbb{QP}^1$. The left version $[\cdot]^{\flat}_q$ was first mentioned in \cite[Remark 3.2]{MGO-2022}, and subsequently studied by Bapat-Becker-Licata in \cite{BBL}. This left version gives a non-standard deformation of integers: for $n \in \mathbb{N}_{\geq 2}$ we get
$$
[n]^{\flat}_q = q^n+q^{n-2}+q^{n-3}+\ldots +q+1 \, .
$$

\smallskip
Let us remark that Definition \ref{def:q-numbers} depends on the $q$-deformation of the representation of the modular group. 

Our goal here is to initiate a discussion on the uniqueness of $q$-rationals and expose a new family of deformations, different from the standard right and left versions of $q$-rationals, but delivering the standard $q$-integers \eqref{Eq:def-n-q}. 
Our main tool is the $\SL_2(\C)$-character variety of the modular group, describing all isomorphism classes of representations. This is why we specialize to $q\in\C$.

\begin{fakethm}[Corollary \ref{Coro:uniqueness-q-deformation}]\label{Fakethm:A}
There are exactly two deformations $\mathbb{QP}^1 \times \mathbb{C}^*\to \mathbb{CP}^1$, $(x,q)\mapsto [x]_q$, which coincide on $\Z$ with the usual $q$-integers $[n]_q=\frac{1-q^n}{1-q}$ and which are equivariant with respect to the modular group $\PSL_2(\Z)$. The modular group acts on the first factor of $\mathbb{QP}^1\times \mathbb{C}^*$ by the usual linear-fractional transformations and trivially on the second factor. It acts on $\mathbb{CP}^1$ via $\varrho_q$.
\end{fakethm}

In Section \ref{Sec:uniqueness-and-chi} we use the character variety approach in order to prove that any non-abelian irreducible representation $ \varrho : \PSL_2(\Z) \rightarrow \PSL_2(\C)$ of the modular group is conjugated to the one given in \eqref{eq:q_representation}, for some $q \in \C$, see Theorem \ref{thm:uniqueness-conj}. 
Even if conjugated, these different representations are still interesting from the combinatorial point of view, as we argue in Section \ref{Sec:q-t}. 
By fixing $\varrho(T)=T_q$, which ensures that $q$-integers stay unchanged, we search for all possible $\varrho(S)$ and obtain a $2$-parameter family of representations $\varrho_{q,t}$ of the modular group. Finally, even though the parameter $t$ may seem unnecessary (since we show that any $\varrho_{q,t}$ is conjugated to $\varrho_{q,1}$, with $\varrho_{q,1}(S)=S_q$), we show that for $t=0$, the left version of integers for the representation $\varrho_{q,0}$ coincides with the classical one. This justifies its detailed combinatorial study. 

The generators of this new representation are:
\begin{equation}\label{eq:def_S_q_0}
T_{q} = \begin{pmatrix} q& 1\\ 0& 1\end{pmatrix}; \;\;\;\;
S_{q,0}=\begin{pmatrix} 1-q& -1\\ q^2-q+1& q-1\end{pmatrix}.
\end{equation}

In Section \ref{Sec:t-0} we study the positivity properties of the numerators and denominators of the corresponding rational fractions in $q$ of this new representation $\varrho_{q,0}$ and give the combinatorial interpretation of their coefficients.
A nice application of this deformation is a direct computation of the Jones polynomial for rational knots.

\smallskip
This work shows that there are only two ways to complete usual $q$-integers to deformations of rationals via the action of the modular group. 
Another direction is to change the symmetry group and to ask for deformations equivariant with respect to another discrete group, acting transitively on $\mathbb{QP}^1$. In upcoming work we will present other ways to deform rationals, preserving the usual $q$-integers, while replacing the modular group by Hecke and Coxeter groups.

\medskip
\noindent \textbf{Acknowledgements.}
We particularly thank Pierre-Louis Blayac, Sophie Morier-Genoud, Julien Marché, Valentin Ovsienko and Christopher-Lloyd Simon for stimulating discussions. 
We also thank Sébastien Labbé, Nicholas Ovenhouse and Valentin Ovsienko for helpful comments on the first version of this paper. We used Claude Code to proof-read the first arxiv preprint of this paper.

This work was supported by the PEPS JCJC 2026 - UMR 5208 (ICJ) by CNRS Mathematics.  P. J. thanks Université de Reims Champagne Ardenne where this work was initiated and Institut Camille Jordan of University Lyon 1 for their hospitality.
O. P.-R. has been supported by the ANR grant GALS ANR-23-CE40-0001. 
A. T. has been supported by a BQR grant of Université Claude Bernard Lyon 1.

%%%%%%%%%%%%%%%%%%%%%%%%%%%%%%%%%%%%%%%%%%%%%%%%%%%%%%%%%%%%%%%%%%%%%%%%%%%%%
\section{Uniqueness of $q$-rationals and character varieties}\label{Sec:uniqueness-and-chi}
%%%%%%%%%%%%%%%%%%%%%%%%%%%%%%%%%%%%%%%%%%%%%%%%%%%%%%%%%%%%%%%%%%%%%%%%%%%%%

We discuss the space of all deformations of homomorphisms of the modular group into $\SL_2(\C)$. For that, the convenient tool is the concept of the character variety.

\subsection{Character varieties}

Character varieties describe the set of all isomorphism classes of representations of a (mostly discrete) group $\Gamma$ into some Lie group $G$. This representation-theoretic concept plays an important role in differential geometry, where representations appear as holonomy maps of geometric structures or monodromies of flat connections.

Here, we restrict to the following setting. Let $\Gamma$ be a finitely generated non-abelian group. We will mostly use $\Gamma=\mathrm{PSL}_2(\Z)$, the modular group. The target Lie group is $G=\mathrm{SL}_2(\C)$.
We then define
$$\chi(\Gamma,\mathrm{SL}_2(\C)):= \mathrm{Hom}(\Gamma,\mathrm{SL}_2(\C))/\mathrm{SL}_2(\C),$$
where we quotient out the conjugacy action. Note that we use a GIT quotient to get a Hausdorff topological space. This means that we identify two conjugacy classes of homomorphisms, if their closures intersect.

Recall that a representation $\varrho:\Gamma\to \mathrm{SL}_2(\C)$ is \emph{reducible} if its image $\varrho(\Gamma)$ fixes some 1-dimensional subspace, and \emph{irreducible} otherwise.
We give some well-known background,
leading to the description of the character variety of the free group $F_2$. Our main references are \cite{Culler-Shalen, Goldman}. 

\smallskip

First, the representations are determined by their character, up to conjugacy. For the proof of this rigidity statement see \cite[Prop. 1.5.2]{Culler-Shalen}.

\begin{proposition}[\cite{Culler-Shalen}]\label{Prop:rep-rigidity}
Let $\varrho, \varrho':\Gamma\to \mathrm{SL}_2(\C)$ be two representations with the same character map. Assume further that $\varrho$ is irreducible. Then $[\varrho]=[\varrho']$ in $\chi(\Gamma,\SL_2(\C))$, i.e. $\varrho$ and $\varrho'$ are conjugated.
\end{proposition}

We restrict now to the case of a group $\Gamma$ with two generators, called $g$ and $h$. The proof of the following proposition can be found in Goldman \cite[Section 2.3]{Goldman}.

\begin{proposition}[\cite{Goldman}]\label{Prop:irrep}
A representation $\varrho:\Gamma\to \mathrm{SL}_2(\C)$ is reducible if and only if we have $$\tr\varrho(ghg^{-1}h^{-1})=2.$$
\end{proposition}

\begin{Remark}
In the general setting, where $\Gamma$ is not restricted to have two generators, a representation $\varrho:\Gamma\to\mathrm{SL}_2(\C)$ is reducible if and only if $\tr\varrho(c)=2$ for all $c\in[\Gamma,\Gamma]$, see \cite[Lemma 1.2.1]{Culler-Shalen}.
\end{Remark}

The conjugacy invariant functions on $\mathrm{SL}_2(\C)$ are generated by the trace. For $[\varrho]\in\chi(\Gamma,\SL_2(\C))$, put
\begin{equation}\label{Eq:trace-coord}
x=\tr\varrho(g) \;,\;\;y=\tr\varrho(h) \;\;\text{ and }\;\;z=\tr\varrho(gh).
\end{equation}

The trace identity $$\tr(A)\tr(B)=\tr(AB)+\tr(AB^{-1})$$
allows to express any trace $\tr(w)$ with $w\in\varrho(\Gamma)$ as a polynomial in $x, y, z$, see \cite{Goldman}. Indeed, write $w$ as a reduced word in $g,h$ and their inverses. If $w=gagb$ (with $a, b$ subwords), we can use the trace identity for $A=ga, B=gb$ to reduce to shorter expressions. If $w=gag^{-1}b$, we can use $A=ga, B=g^{-1}b$. The same holds for $w=hahb$ or $w=hah^{-1}b$. This leads to a word which has at most one element out of $\{g,g^{-1}\}$ and out of $\{h,h^{-1}\}$. Finally, we note that $\tr(\varrho(g))=\tr(\varrho(g^{-1}))$. 

A more detailed algorithm can be found in \cite[Section 2.2]{Goldman_Fricke}.

\begin{example}
Let us compute the trace of the commutator $\tr [g, h]$. For a simpler writing, we omit the representation $\varrho(\cdot)$ in this example.

We use the trace identity with $A=gh$ and $B=g^{-1}h^{-1}$:
$$\tr [g, h]=\tr(ghg^{-1}h^{-1}) = \tr(gh)\tr(hg)-\tr(ghhg)=z^2-\tr(g^2h^2).$$
Further, 
$$\tr(g^2h^2)=\tr(g^2h)\tr(h)-\tr(g^2hh^{-1})=y\tr(g^2h)-\tr(g^2).$$
Then $\tr(g^2)=\tr(g)^2-\tr(gg^{-1})=x^2-2$ and $$\tr(g^2h)=\tr(hg)\tr(g)-\tr(hgg^{-1})=xz-y.$$ Putting all together, we get
$$\tr(ghg^{-1}h^{-1}) = x^2+y^2+z^2-xyz-2.$$
\end{example}

\begin{proposition}
Consider $\varrho:\Gamma\to\SL_2(\C)$ and put $k=\tr\varrho(ghg^{-1}h^{-1})$ and $x,y,z$ as in Equation \eqref{Eq:trace-coord}.

If $k\neq 2$, then any $\varrho'$ with trace coordinates $(x,y,z)$ is conjugate to $\varrho$.

If $k=2$, then there is a representation $\varrho_0$ such that the closure of $\SL_2(\C).\varrho_0$ contains all other representations with $k=2$ and trace coordinates $(x,y,z)$.
\end{proposition}

The proof is a combination of the rigidity of characters for irreducible representations, and the characterization of reducible representations.

\begin{proof}
Consider first the case where $k\neq 2$. By Proposition \ref{Prop:irrep}, this implies that $\varrho$ is irreducible. Take $\varrho'$ with the same coordinates $(x,y,z)$. By the trace identities, this implies that $\tr(\varrho(c))=\tr(\varrho'(c))$ for any $c\in\Gamma$. By Proposition \ref{Prop:rep-rigidity}, this implies that $\varrho$ and $\varrho'$ are conjugated.

Second, consider the case where $k=2$. Then $\varrho$ is reducible. Hence we can conjugate $\varrho$ such that
\begin{equation}\label{Eq:aux-rho}
\begin{aligned}
\varrho(g)=\begin{pmatrix} a & b \\ 0 & a^{-1}\end{pmatrix}\;\;\text{ and }\;\;\varrho(h)=\begin{pmatrix} c & d \\ 0 & c^{-1}\end{pmatrix},\\
\text{ where } \; \;  x=a+a^{-1}, y=c+c^{-1} \; \text{ and } \; \;  z=ac+(ac)^{-1}.
\end{aligned}
\end{equation}

Given $(x,y,z)$, these equations determine $(a,c)$ up to $(a^{-1},c^{-1})$, which is conjugate to $(a,c)$ using a permutation matrix.

Define the representation $\varrho_0$ by its values on generators :
$$\varrho_0(g)=\begin{pmatrix} a & 0 \\ 0 & a^{-1}\end{pmatrix}\;\;\text{ and }\;\;\varrho_0(h)=\begin{pmatrix} c & 0 \\ 0 & c^{-1}\end{pmatrix}.$$
Using a conjugacy by the diagonal matrix $g_t=\mathrm{diag}(e^{-t}, e^{t})$, we see that $g_t\varrho g_t^{-1}$ tends to $\varrho_0$ when $t$ tends to $\infty$. Hence $\varrho$ belongs to the closure of the conjugacy orbit of $\varrho_0$. Since any representation with $k=2$ and trace coordinates $(x,y,z)$ can be put into the form \eqref{Eq:aux-rho}, they are all in the closure of the orbit of $\varrho_0$.
\end{proof}

A direct consequence of the above proposition gives the following well-known result (see for instance \cite[Theorem A]{Goldman} and references therein):
\begin{coro}\label{Coro:chi-F2}
The character variety of the free group in two generators $F_2=\langle g,h\rangle$ is
$$\chi(F_2, \SL_2(\C))\cong \C^3,$$
where the isomorphism is given by $[\varrho]\mapsto (\tr\varrho(g),\tr\varrho(h),\tr\varrho(gh))$.
\end{coro}

In the applications to $q$-rationals, we will consider the modular group $\Gamma=\PSL_2(\Z)$, and representations into $\PSL_2(\C)$. We will need the following simple fact:

\begin{proposition}\label{Prop:lift-order}
Let $n\geq 2$ be even. If $A^n=1$ in $\PSL_2(\C)$, where  $n$ is the minimal order of $A$, then $\tilde{A}^n = -1$ in $\SL_2(\C)$ for any lift $\tilde{A}$ of $A$.
\end{proposition}
\begin{proof}
Suppose that $\tilde{A}^n=1$ in $\SL_2(\C)$. Since the minimal polynomial of $\tilde{A}$ divides $X^n-1$ which has simple roots, $\tilde{A}$ is diagonalisable. For $\tilde{A}=\mathrm{diag}(\lambda,\lambda^{-1})$, one representative of $A$ is given by $\mathrm{diag}(\lambda^2,1)$, which is of order $n/2$, a contradiction. Therefore $\tilde{A}^n=-1$.
\end{proof}

\subsection{Uniqueness of $q$-rationals}

We are now ready to address the question of uniqueness of $q$-rationals. The main ingredient in Definition \ref{def:q-numbers} of $q$-rationals is the representation of the modular group inside $\PSL_2(\C)$. Classifying deformations of rational numbers equivariant with respect to the modular group is then equivalent to classifying representations of $\PSL_2(\Z)$ into $\PSL_2(\C)$.

The following shows that $q$-rationals are unique up to conjugacy. Despite this, the combinatorial aspects of $q$-rationals may be interesting to study for other representations, as we will see in the subsequent section.

\begin{theorem}\label{thm:uniqueness-conj}
Let $\varrho:\PSL_2(\Z)\to\PSL_2(\C)$ be a non-abelian irreducible representation of the modular group. Then 
it is equivalent, up to conjugacy, to one of the representations $\varrho_q$ from \eqref{eq:q_representation} for a certain $q\in \C^*$.
\end{theorem}

The proof idea is that the modular group is generated by two elements, $S$ and $T$. Hence its character variety is described by trace coordinates. The relations among $S$ and $T$ fix all but one of these coordinates (the last coordinate being a function of $q$).

\begin{proof}
Consider a representation $\varrho:\PSL_2(\Z)\to\PSL_2(\C)$. Recall the presentations
\begin{align*}
    \PSL_2(\Z)&=\langle S,T\mid S^2=(TS)^3=1\rangle, \\
    \SL_2(\Z)&=\langle \tilde{S},\tilde{T}\mid \tilde{S}^4=1, (\tilde{T}\tilde{S})^3=\tilde{S}^2\rangle.
\end{align*}

We lift $\varrho(S)$ and $\varrho(T)$ to elements $\tilde{S}$ and $\tilde{T}$ in $\SL_2(\C)$, unique up to sign.
Since $\varrho(S)^2=1$, we get $\tilde{S}^2=\pm 1$. By Proposition \ref{Prop:lift-order}, we see that if $\tilde{S}^2= 1$, then $\varrho(S)=1$, hence $\varrho$ is abelian, which is excluded by the hypothesis. Hence $\tilde{S}^2=- 1$.

Since $\varrho(TS)^3=1$, we get $(\tilde{T}\tilde{S})^3=\pm 1$. Changing the sign of $\tilde{T}$ if necessary, we can impose $(\tilde{T}\tilde{S})^3=-1=\tilde{S}^2$. Hence we can extend $\tilde{S}$ and $\tilde{T}$ to a homomorphism $\tilde{\varrho}:\SL_2(\Z)\to\SL_2(\C)$. We will write $\tilde{A}$ as abbreviation for $\tilde{\varrho}(A)$, where $A\in\SL_2(\Z)$.

By Proposition \ref{Prop:rep-rigidity} and Corollary \ref{Coro:chi-F2}, we know $\tilde{\varrho}$ is described by the trace coordinates
$(\tr(\tilde{S}),\tr(\tilde{T}), \tr(\widetilde{TS})).$
Since $\tilde{S}^2=-1$, we get $\tr\tilde{S}=0$. From $(\tilde{T}\tilde{S})^3=-1$ and $x^3+1=(x+1)(x^2-x+1)$, we have either $\widetilde{TS}=-1$ or the characteristic polynomial of $\widetilde{TS}$ is $x^2-x+1$. In the first case, $\varrho$ is abelian, and in the second, we get $\tr\widetilde{TS}=1$.

Therefore the representation $\tilde{\varrho}$ is characterized by $\tr \tilde{T}$. For $\tilde{T}=\left(\begin{smallmatrix}q^{1/2}&q^{-1/2}\\ 0& q^{-1/2}\end{smallmatrix}\right)$, the normalized version of $T_q$, the trace is given by $q^{1/2}+q^{-1/2}$. This function for $q\in\mathbb{C}^*$ is surjective on $\C$. Therefore, these deformations are the only ones inside $\PSL_2(\C)$ up to conjugacy.
\end{proof}

\begin{Remark}
For a specialisation to real numbers, the $\PSL_2(\R)$-character variety of the modular group has been described by Simon in \cite[Section 3.2]{Simon} and also in his PhD thesis \cite[Chapter 5]{Simon-thesis}.
\end{Remark}

We remark that the representations $\varrho_q$ are all non-abelian since $T_qS_q\neq S_qT_q$. By Proposition \ref{Prop:irrep}, we see that $\varrho_q$ is reducible if and only if $2=\mathrm{tr}(S_qT_qS_q^{-1}T_q^{-1}) = \tfrac{q^2+q+1}{q}$, i.e. if and only if $q^2-q+1=0$. Put $\sigma=\tfrac{1+i\sqrt{3}}{2}$, one root of $q^2-q+1=0$. Then $\varrho_q$ is irreducible for $q\notin\{\sigma,\bar\sigma\}$.

\begin{coro}
The representations $\varrho_q$ are conjugated for the parameters $q$ and $q^{-1}$ for $q\in\mathbb{C}^*\backslash\{\sigma,\bar\sigma\}$.
\end{coro}
This has already been noticed in the PhD thesis of Simon \cite[Corollary 5.7]{Simon-thesis}.

\begin{proof}
As seen above, the only invariant is $\mathrm{tr}(T_q) = q^{1/2}+q^{-1/2}$, which does not change when $q$ becomes $q^{-1}$. By uniqueness, the two representations have to be conjugated.
\end{proof}

\begin{Remark}
To be explicit, the conjugacy matrix, known as the \emph{transition map} from \cite[Section 2.2]{Thomas} and \cite[Section 1.3]{Jouteur}, is given by
$$g_q = \begin{pmatrix}q & 1-q \\ q-1 & 1\end{pmatrix}.$$
The determinant of $g_q$ is $q^2-q+1$, which is non-zero for $q\notin\{\sigma,\bar\sigma\}$.
\end{Remark}

This viewpoint allows a very simple proof of the palindromicity property of traces of $q$-deformed elements of $\PSL_2(\Z)$, first noticed in \cite[Theorem 3]{Leclere-Morier-Genoud}. We call a Laurent polynomial $P$ \emph{palindromic} if $P(q)=q^{N}P(q^{-1})$ for some integer $N$.

\begin{coro}
Let $A_q\in\varrho_q(\mathrm{PSL}_{2}(\Z))$. Then $\mathrm{tr}(A_q)$ is palindromic in $q$.
\end{coro}

\begin{proof}
Since the representations for parameter $q$ and $q^{-1}$ are conjugated, we know that $A_q$ is conjugated to $A_{q^{-1}}$, up to some power of $q$, since all matrices are considered projectively. Since the trace is invariant under conjugacy, we get $\mathrm{tr}(A_q) = q^N\mathrm{tr}(A_{q^{-1}})$ for some integer $N$. This proves palindromicity.
\end{proof}

Note that if we use normalized matrices (with determinant 1), we get exact palindromicity (with $N=0$).

\subsection{Formal uniqueness of $q$-rationals.}

Here we give a formal version of previously proven Theorem \ref{thm:uniqueness-conj} on uniqueness of $q$-rationals. We thank Christopher-Lloyd Simon for his insights and references.

\smallskip

Let $R$ be a commutative unital ring of characteristic different from $2$. For a finitely generated group $\Gamma$, the character variety is defined to be
$$\chi(\Gamma,\SL_2(R)) = \mathrm{Spec}(R[\mathrm{Hom}(\Gamma,\SL_2(R))]^{\mathrm{PGL}_2(R)})\, .$$

For the free group $\Gamma=F_2=\langle a, b\rangle$, the class of an irreducible representation $[\varrho]$ in $\chi(F_2,\SL_2(R))$ is still uniquely given by its trace coordinates $(\tr\varrho(a), \tr\varrho(b), \tr\varrho(ab))$, as proven by Saito in \cite{Saito}. This means that if two irreducible representations $\varrho_1$ and $\varrho_2$ have the same trace coordinates, then they are conjugate by an element of $\PGL_2(R)$.

We then have the following:
\begin{theorem}
Let $R$ be a commutative unital ring, $\mathrm{char} R \neq 2$. Consider a non-abelian irreducible representation $\varrho:\PSL_2(\Z)\to\PSL_2(R)$. Then there is an element $q$ in a quadratic extension of $R$, such that $\varrho$ is conjugated to $\varrho_q$ from \eqref{eq:q_representation} under $\PGL_2(R)$.
\end{theorem}

A particular choice for the ring is $R=\Z[x,x^{-1}]$, which gives a formal version of Theorem \ref{thm:uniqueness-conj}. 
The proof is very similar to the one of that theorem. Some quadratic equation has to be solved, which explains the necessity to pass to a quadratic extension.

\begin{proof}
 As in proof of Theorem \ref{thm:uniqueness-conj}, we lift $\varrho(S)$ and $\varrho(T)$ to $\widetilde{S}$ and $\widetilde{T}$ in $\SL_2(R)$ such that $\widetilde{S}^2=-1$ (since $\varrho$ is non-abelian and $\mathrm{char}(R)\neq 2$) and $(\widetilde{T}\widetilde{S})^3 = -1$. We then have a homomorphism $\widetilde{\varrho}: \SL_2(\Z)\to \SL_2(R)$.
 
 Such a homomorphism is uniquely determined (up to $\PGL_2(R)$-conjugacy) by its trace coordinates $(\tr(\widetilde{S}), \tr(\widetilde{T}), \tr(\widetilde{T}\widetilde{S}))$. By the same arguments as above, we see that only $\tr(\widetilde{T}):=t$ is a free parameter. 
 
 To find the adequate parameter $q$, we need $t=\tr(T_q)=q^{1/2}+q^{-1/2}$, which is equivalent to a quadratic equation in $q$. Hence, there is a solution $q$ in a quadratic extension of $R$, which makes $\varrho$ conjugate to $\varrho_q$.
\end{proof}

\begin{Remark}
If one only considers conjugacy by $\PSL_2(R)$, then there are exactly two possible choices which correspond to $T_q$ and $T_q^{-1}$ \cite{Simon-2}.
\end{Remark}

%%%%%%%%%%%%%%%%%%%%%%%%%%%%%%%%%%%%%%%%%%%%%%%%%%%%%%%%%%%%%%%%%%%%%%%%%%%%%
\section{Two-parameter family of deformations}\label{Sec:q-t}
%%%%%%%%%%%%%%%%%%%%%%%%%%%%%%%%%%%%%%%%%%%%%%%%%%%%%%%%%%%%%%%%%%%%%%%%%%%%%

Although the deformation of $q$-rationals is unique up to conjugacy, there are interesting representatives, other than the one introduced by Morier-Genoud and Ovsienko.

Note that the standard map $T_q$ defined by the Equation \eqref{eq:q_representation} is the only map in $\PGL_2 (\mathbb{Z} (q))$ such that it preserves the set of $q$-integers and acts on it as a shift map:
$$ T_q [n]_q= q[n]_q+1 = [n+1]_q \; \; \text{for all} \; \; n \in \Z.$$ 
The uniqueness of $T_q$ follows from the fact that a M\"obius transformation is determined from 3 values. 
The set $\{[n]_q\mid n\in\Z\}$ is infinite for generic $q$ (the only exception is $q=-1$, where only 2 values occur). 

Consider a non-abelian irreducible representation $\varrho:\mathrm{PSL}_2(\Z)\to\PSL_2(\C)$. Using conjugacy we can impose $\varrho(T)=T_q$. Let us now search for the explicit form of $\varrho(S)$, where $T$ and $S$ are the usual generators of the modular group. 

\begin{theorem}\label{thm:uniqueness}
 Consider a representation of the modular group via rational functions $\varrho_q:\PSL_2(\Z) \to \PGL_2(\Z(q))$, such that $\varrho_q(T)=T_q$ and $\varrho_1(S)=S$. Then there is a one-parameter family of possible representations that can be made explicit by
\begin{equation}\label{Eq-Sqt}
\varrho_{q,t}(S)=S_{q,t}=\begin{pmatrix}(t-1)(q-1) & -\tfrac{1}{q}(q+(q-1)^2t^2)\\ q^2-q+1 & (t-1)(1-q)\end{pmatrix},
\end{equation}
where $t$ is a formal parameter such that $t\neq -\tfrac{q}{(1-q)^2}$. The standard $S_q$ as in \eqref{eq:q_representation} is recovered when $t=1$.
\end{theorem}

\begin{proof}
This is an explicit computation. We search for $\varrho(S)= \left(\begin {smallmatrix}a&b\\ c&d \end{smallmatrix}\right)$ satisfying $\varrho(S)^2=(T_q \varrho(S))^3=1$. Second equation implies $\varrho(S) \neq 1$. The first equation is then equivalent to $d=-a$. If $a=0$, then $S_q$ is the only solution. If not, we can rescale $\varrho(S)$ such that $a=q-1$, since we have $a=0$ for $q=1$.

The second equation then gives a quadratic expression for $b$ as function of $c$ and $q$, which factorizes. One root is $b=\frac{q-1}{q}$, which does not give the correct value for $q=1$. The second root is
$$b=-\frac{1}{cq}((q-1)^2(q^2-q+1)+2c(q-1)^2+c^2).$$
Finally, we put $t=\frac{q^2-q+1}{c}+1$ to get the form \eqref{Eq-Sqt}. Note that  $\det S_{q,t}\neq 0$ for $t\neq -\tfrac{q}{(1-q)^2}$.
\end{proof}

By Theorem \ref{thm:uniqueness-conj}, the representation $\varrho_{q,t}$ is conjugated to $\varrho_q$. The conjugacy is explicitly given in the following

\begin{proposition}\label{Prop:rho-q-t-conj}
We have $\varrho_{q,t} = U_f\varrho_{q,1}U_f^{-1},$
where $$U_f=\left(\begin{smallmatrix}f & \tfrac{1-f}{1-q}\\ 0 & 1\end{smallmatrix}\right), \;\text{ with }\;f=\frac{q+(q-1)^2t}{q^2-q+1} .$$
\end{proposition}

Note that $U_f$ is a homothety of factor $f$ around $\tfrac{1}{1-q}$. In particular $U_f$ preserves the two fixed points of $T_q$, $[\infty]_q^\sharp$ and $[\infty]_q^\flat$. Note further that $\det U_f=f\neq 0$ since $t\neq -\tfrac{q}{(1-q)^2}$.

\begin{proof}
We have $U_fT_qU_f^{-1}=T_q$, which geometrically comes from the fact that $U_f$ and $T_q$ share the same fixed points. An explicit computation also shows that
$U_fS_qU_f^{-1}~=~S_{q,t}.$
\end{proof}

The representation $\varrho_{q,t}$  allows to define a 2-parameter family of deformations of rationals:
\begin{definition}
Let $\tfrac{r}{s}\in\mathbb{QP}^1$. There is $M\in\PSL_2(\Z)$ such that $\tfrac{r}{s}=M(\infty)$. Then define the right and left $(q,t)$-deformation of $\tfrac{r}{s}$ by
$$\left[\frac{r}{s}\right]_{q,t}^\sharp = \varrho_{q,t}(M)(\infty)\;\;\text{ and }\;\; \left[\frac{r}{s}\right]_{q,t}^\flat = \varrho_{q,t}(M)\left(\frac{1}{1-q}\right).$$
For $t=1$, we recover Definition \ref{def:q-numbers}, with $[x]_{q,1}^\square=[x]_q^\square$, where $\square\in\{\flat, \sharp\}$.
\end{definition}

\begin{example} Here are some examples, giving a flavor of $(q,t)$-rationals:

\begin{itemize}
\setlength\itemsep{0.5em}
    \item Right integers: $[n]_{q,t}^\sharp = \frac{q^n(q-1)t+1+q^2+q^3+...+q^n}{q^2-q+1}=\frac{q^n(q-1)(t-1)}{q^2-q+1}+[n]_q$,
    \item Left integers: $[n]_{q,t}^{\flat} = q^{n-1}(q-1)t+[n]_q$,
    \item Right inverses of integers: $[\tfrac{1}{n}]^{\sharp}_{q,t} = \tfrac{1+t(q-1)[n]^{\flat}_q}{(q^2-q+1)[n]_q}$,
    \item $\left[\frac{5}{3}\right]_{q,t}^\sharp = \frac{[5]_q + qt(q-1)(1+q^2+q^3)}{(1-q+q^2)(1+q+q^2)}$ and $\left[\frac{5}{3}\right]_{q,t}^\flat = \frac{1+2q+q^2+q^3 + qt(q-1)(1+q+q^2)}{1+q+q^3}$.
\end{itemize}
\end{example}

As a consequence of Proposition \ref{Prop:rho-q-t-conj}, we get:
\begin{proposition}\label{Prop:q-t-from-q-1}
For all $x\in\mathbb{QP}^1$ and $\square\in\{\flat, \sharp\}$ we have
$$[x]_{q,t}^\square = U_f[x]_{q}^\square,$$
where $U_f$ is given as in Proposition \ref{Prop:rho-q-t-conj}
\end{proposition}
\begin{proof}
Let $x=T^{a_0}ST^{a_1}S \cdots T^{a_n}S(\infty)$ be the continued fraction representation of $x$, where $a_0\in\Z$ and $a_i\in \N_{>0}$ for all $i>0$.
Then
\begin{align*}
    [x]_{q,t}^\sharp &= T_{q,t}^{a_0}S_{q,t}T_{q,t}^{a_1}S_{q,t} \cdots T_{q,t}^{a_n}S_{q,t}(\infty) \\
    &= U_f T_q^{a_0}S_qT_q^{a_1}S_q\cdots T_q^{a_n}S_q U_f^{-1}(\infty)\\
    &= U_f [x]_q^\sharp,
\end{align*}
where we used that $U_f(\infty)=\infty$. The same computation for $[x]_{q,t}^\flat$ works, where we replace $\infty$ by $\tfrac{1}{1-q}$.
\end{proof}

\begin{Remark}
The previous proposition allows to define $[x]_{q,t}$ for irrational $x$, using $[x]_{q,1}$ defined in \cite{MGO-2022}.
\end{Remark}

\begin{Remark}\label{Rk:conj-neg-matrix}
The extended $\PGL_2(\Z)$-symmetry \cite{Jouteur} and the transition map $q~\mapsto~q^{-1}$ \cite{Jouteur, Thomas} also pass to the $(q,t)$-rationals, via conjugacy by $U_f$. In particular, the negation $N = \left(\begin{smallmatrix}
-1 & 0 \\
0 & 1\\
\end{smallmatrix}\right)$ is deformed as 
$$N_{q,t}=\begin{pmatrix} -1 & \frac{(q-1)(q(2t-1)+(q-1)^2t^2)}{q(q^2-q+1)}\\ q-1 & 1 \end{pmatrix}.$$
\end{Remark}

The definition of the $q$-numbers by Morier-Genoud and Ovsienko \cite{MGO-2020} was motivated by the fact that the orbit $T_q^n(0)=T_q^n S_q (\infty) = [n]_q$ realizes $q$-deformations of integers, important in combinatorics. 
The left version of Bapat-Becker-Licata \cite{BBL} lacks this property. Still, there is another $q$-version of rationals coinciding with $q$-integers:

\begin{coro}\label{Coro:uniqueness-q-deformation}
There are exactly two deformations $\mathbb{QP}^1 \times \mathbb{C}^*\to \mathbb{CP}^1$, $(x,q)\mapsto [x]_q$, which coincide on $\Z$ with the usual $q$-integers $[n]_q=\frac{1-q^n}{1-q}$ and which are equivariant with respect to the modular group $\PSL_2(\Z)$. The modular group acts on the first factor of $\mathbb{QP}^1\times \mathbb{C}^*$ by the usual linear-fractional transformations and trivially on the second factor. It acts on $\mathbb{CP}^1$ via $\varrho_q$.
\end{coro}

\begin{proof}
It follows from Theorem \ref{thm:uniqueness} that $S_{q,t} [\infty]_{q}^{\sharp} = 0$ if and only if $t=1$, and $S_{q,t} [\infty]_{q}^{\flat}~=~0$ if and only if $t=0$ (or $t=-\tfrac{q}{(1-q)^2}$ which is excluded since it makes $S_{q,t}$ degenerate). This means that classical $q$-deformations of integers appear as the orbit of $T_q$ in two definitions of rational $(q,t)$-numbers :  as the right orbit of $0$ for $t=1$ (standard case), or the left orbit of $0$ for $t=0$ (new case). 

Moreover, from Proposition \ref{Prop:rho-q-t-conj} it follows that there is a homothety which maps all the new left $q$-integers $[n]_q^\flat$ to the standard $q$-integers $[n]_q^\sharp$.
\end{proof}

%%%%%%%%%%%%%%%%%%%%%%%%%%%%%%%%%%%%%%%%%%%%%%%%%%%%%%%%%%%%%%%%%%%%%%%%%%%%%%
\section{A new family of deformed rationals}\label{Sec:t-0}
%%%%%%%%%%%%%%%%%%%%%%%%%%%%%%%%%%%%%%%%%%%%%%%%%%%%%%%%%%%%%%%%%%%%%%%%%%%%%

Here we study the properties of $(q,t)$-rationals for $t=0$, in particular their positivity and combinatorics.

Corollary \ref{Coro:uniqueness-q-deformation} 
 shows that the study of the operator $S_{q,0}$ defined by \eqref{eq:def_S_q_0} as 
 \begin{equation*}
 S_{q,0}=\begin{pmatrix} 1-q& -1\\ q^2-q+1& q-1\end{pmatrix}.
\end{equation*}
 
 is equally justified from the combinatorics point of view as the operator $S_{q}$ defined by \eqref{eq:q_representation}. Using the extension to matrices of negative determinant, the $(q,0)$-rationals are characterized by the following relations 
 \[
 [x+1]_{q,0}^\square = q[x]_{q,0}^\square+1 \text{ and } \left[\frac{1}{x}\right]_{q,0}^\square = \frac{1}{(q^2-q+1)[x]_{q,0}^\triangle},
 \]
for $\square, \triangle \in \{\sharp,\flat\}$, $\square \neq \triangle$ and initial conditions $[\infty]_{q,0}^\sharp=\tfrac{1}{0}$, $[\infty]_{q,0}^\flat=\tfrac{1}{1-q}$.

Many questions arise on the combinatorial sense of the coefficients of rational fractions arising in this new deformation of rationals, analogous to those treated already for the case $t=1$. 
For $\tfrac{a}{b}\in\mathbb{QP}^1$, we write
$$
\left[\frac{a}{b}\right]^{\square}_{q,t}=\frac{A_{t}^{\square}(q)}{B_{t}^{\square}(q)},\;\; \square \in \{\sharp, \flat\},
$$
with $A_{t}^{\square}$ and $B_{t}^{\square}$ being coprime polynomials in $q$. For $t=0$ and $t=1$, the only cases considered here, we normalize $B_{t}^{\square}$ to have a positive leading coefficient, except for $\tfrac{1}{0}$, where we put $[\tfrac{1}{0}]_{q,t}^\flat=\tfrac{1}{1-q}$.
These polynomials depend on $\tfrac{a}{b}$ but we omit this dependence to lighten the notations. 
%In the following, the subscript $\{q,1\}$ corresponding to a standard deformation is omitted. 

\subsection{Positivity}

At $t=0$, deformed rationals still enjoy positivity properties. Let us start with a list of examples:

\begin{itemize}
 \setlength\itemsep{0.5em}
    \item Integers: $[n]_{q,0}^\sharp = \frac{1+q^2+q^3+...+q^n}{q^2-q+1}$ and $[n]_{q,0}^\flat = [n]_q$,
    \item Inverses of integers: $\left[\frac{1}{n}\right]^{\sharp}_{q,0} = \tfrac{1}{(q^2-q+1)[n]_q}$ and $\left[\frac{1}{n}\right]^{\flat}_{q,0} = \tfrac{1}{1+q^2+q^3+...+q^n}$,
    \item $\left[\frac{5}{3}\right]_{q,0}^\sharp = \frac{[5]_q}{(q^2-q+1)(1+q+q^2)}$ and $\left[\frac{5}{3}\right]_{q,0}^\flat = \frac{1+2q+q^2+q^3}{1+q+q^3}$.
\end{itemize}

\medskip

Note that not all denominators of these deformed rationals are positive. However, the positivity property still holds in the following sense: 

\begin{theorem}\label{Thm:positivity-t0}
Let $\tfrac{a}{b}\in\Q_{>0}$. We have the following positivity properties
\begin{itemize}
    \item[(i)] $A_{0}^\flat, B_{0}^\flat\in\N[q]$,
    \item[(ii)]$A_{0}^\sharp\in\N[q]$ and $B_{0}^\sharp \in (q^2-q+1)\N[q]$.
\end{itemize}
\end{theorem}

This theorem is a special case of a statement about deformed Farey determinants.

\begin{definition}\label{Def:q-Farey-det}
Let $t\in\{0,1\}$. For $(\tfrac{a}{b},\tfrac{c}{d})$ any pair of rationals with $(q,t)$-deformation $\left(\tfrac{A^\square_t}{B_t^\square},\tfrac{C_t^\triangle}{D_t^\triangle}\right)$ with $\square, \triangle\in\{\sharp,\flat\}$, we define the \emph{deformed Farey determinant} by
$$d_t^{\square\triangle}\left(\frac{a}{b},\frac{c}{d}\right)= A^\square_t D_t^\triangle-B_t^\square C_t^\triangle,$$
where $(\square,\triangle)\neq (\flat ,\flat)$ for $t=1$ and $(\square,\triangle)\neq (\sharp ,\sharp)$ for $t=0$.
Else, we define
$$d_t^{\square\square}\left(\frac{a}{b},\frac{c}{d}\right)=\frac{A_t^\square D_t^\square-B_t^\square C_t^\square}{q^2-q+1},$$
where $\square=\flat$ for $t=1$ and $\square=\sharp$ for $t=0$.
\end{definition}

\begin{theorem}\label{Thm:pos-Farey-d}
If $\tfrac{a}{b}>\tfrac{c}{d}$, then the four $(q,0)$-deformed Farey determinants are polynomial with positive coefficients:
\[
d_0^{\square\triangle}\left(\frac{a}{b},\frac{c}{d}\right) \in \N[q],
\]
with $\square,\triangle\in \{\sharp,\flat\}$.
\end{theorem}

In the proof of this Theorem, we will use the following two statements already proven in \cite{JPaRoT2026} that use the $(q,1)$-deformed Farey determinants.

\begin{lemma}[Lemma 4.2 in \cite{JPaRoT2026}]\label{Lem:special-values}
Denote by $\mathbb{U}_6$ the set of 6-th roots of unity, and put $\sigma=\tfrac{1+i\sqrt{3}}{2}$, one root of $q^2-q+1=0$. For all $\frac{a}{b}\in\mathbb{Q}$, we have the following:
$$B_1^\sharp (\sigma)+(\sigma-1)A_1^\sharp (\sigma) \in \mathbb{U}_6 \;\;\text{ and }\;\; B_1^\flat(\sigma)+(\sigma-1)A_1^\flat(\sigma)=0 \,.$$
\end{lemma}

\begin{proposition}[Proposition 4.6 in \cite{JPaRoT2026}]\label{Prop:positivity-q-dF}
For $\tfrac{a}{b}> \tfrac{c}{d}$, we have $d_1^{\square\triangle}(\tfrac{a}{b},\tfrac{c}{d})\in\mathbb{N}[q]$ for $\square,\triangle~\in~\{\sharp ,\flat\}$.
\end{proposition}

\smallskip

\begin{proof}[Proof of Theorem \ref{Thm:pos-Farey-d}]
We use the following notation along the proof: for two polynomials $A,B \in \Z[q]$, we write  $A\equiv_q B$ if $A(q)=q^N B(q)$ for some $N \in \Z$. 

By Proposition \ref{Prop:q-t-from-q-1}, we know that $[x]_{q,t}^\square = U_f[x]_{q,1}^\square$ for $\square\in\{\sharp, \flat\}$. For $t=0$ we get 
\begin{equation}\label{Eq:rel-t-0}
[x]_{q,0}^\square = f \frac{A_1^\square}{B_1^\square}+\frac{1-f}{1-q} = \frac{qA_1^\square+ (1-q)B_1^\square}{(q^2-q+1)B_1^\square},
\end{equation}
where we used that $f=\tfrac{q}{q^2-q+1}$. This fraction may not be reduced.\\ Indeed, $\sigma A^\square_1(\sigma) + (1-\sigma)B_1^\square(\sigma) = \sigma^{-1}((\sigma-1)A_1^\square(\sigma) + B^\square_1(\sigma))$. By Lemma \ref{Lem:special-values}, for $\square = \sharp$ the greatest common divisor is either $1$ or $q$. But for $\square = \flat$, the greatest common divisor is either $q^2-q+1$ or $q(q^2-q+1)$. Therefore,
\[
A_0^\sharp \equiv_q qA^\sharp_1 + (1-q)B^\sharp_1 \, \text{ , } \, B_0^\sharp \equiv_q (q^2-q+1)B_1^\sharp,
\]
\[
A_0^\flat \equiv_q \frac{qA^\flat_1 + (1-q)B^\flat_1}{q^2-q+1} \, \text{ , } \, B_0^\flat \equiv_q B_1^\flat. 
\]

\noindent The $(q,0)$-deformed Farey determinants can then be expressed in terms of the $(q,1)$-deformed Farey determinants. 
\begin{align*}
d_0^{\flat\sharp}\left(\frac{a}{b},\frac{c}{d}\right) &= A_{0}^\flat D_{0}^\sharp- B_{0}^\flat C_{0}^\sharp\\
&\equiv_q \frac{qA_1^\flat+(1\!-\!q)B_1^\flat}{q^2-q+1}(q^2\!-\!q\!+\!1)D_1^\sharp-B_1^\flat(qC_1^\sharp+(1\!-\!q)D_1^\sharp) \\
&\equiv_q A_1^\flat D_1^\sharp-C_1^\sharp B_1^\flat \\
&= d_1^{\flat\sharp}\left(\frac{a}{b},\frac{c}{d}\right).
\end{align*}
We then conclude by Proposition \ref{Prop:positivity-q-dF}. Similarly, $d_0^{\sharp\flat} \equiv_q d_1^{\sharp\flat}$ and $d_0^{\flat\flat} \equiv_q d_1^{\flat\flat}$. Finally,
\begin{align*}
d_0^{\sharp\sharp}\left(\frac{a}{b},\frac{c}{d}\right) &= \frac{A_{0}^\sharp D_{0}^\sharp- B_{0}^\sharp C_{0}^\sharp}{q^2-q+1} \\
&\equiv_q \frac{(qA_1^\sharp+(1\!-\!q)B_1^\sharp)(q^2\!-\!q\!+\!1)D_1^\sharp-(q^2-q+1)B_1^\sharp(qC_1^\sharp+(1\!-\!q)D_1^\sharp)}{q^2-q+1} \\
&\equiv_q A_1^\sharp D_1^\sharp-C_1^\sharp B_1^\sharp \\
&= d_1^{\sharp\sharp}\left(\frac{a}{b},\frac{c}{d}\right) \in \N[q].
\end{align*}
\end{proof}

\begin{proof}[Proof of Theorem \ref{Thm:positivity-t0}]
Applying Theorem \ref{Thm:pos-Farey-d} with $\frac{a}{b}$ and $0$, using that $[0]_{q,0}^\flat = 0$, we get the positivity property for $A_0^\sharp$ and $A_0^\flat$ immediately. Then applying the same theorem with $\frac{1}{0}$ and $\frac{a}{b}$, using that $\left[\frac{1}{0}\right]_{q,0}^\sharp = \frac{1}{0}$, we get the denominators.
\end{proof}

\subsection{Combinatorics of $(q,0)$-rational numbers}

The combinatorial model we propose for $(q,0)$-rationals mimicks the already known models for $(q,1)$-rationals. More precisely, we associate to each positive rational number a fence poset. There are slightly different versions to do so in the literature, and we choose here to follow the recent version of Aval, Labbé \cite{Aval_labbe}.

\begin{definition}[\cite{Aval_labbe}]
Let $a = (a_1,a_2,\cdots,a_n)$ be a finite sequence of positive integers. Denote by $N = a_1+a_2+\cdots +a_n$. The fence poset $F(a)$ associated with this sequence is the poset with $N+1$ elements $v_0$, $v_1$, ..., $v_{N}$ and cover relations 
$$
v_0 < v_1 < \cdots < v_{a_1} > v_{a_1+1} > \cdots > v_{a_1+a_2} < \cdots \dots < v_{N-a_n} > \cdots > v_{N}.
$$
We allow also in this definition the case $a_1 = 0$, and in that case the poset $F(0,a_2,\cdots,a_n)$ begins with a decreasing chain.
\end{definition}

\begin{proposition}
Let $x \in \Q_{>0}$. 
\begin{itemize}
    \item[(i)] Denote by $x = [a_1,a_2,\cdots,a_{2n}]$ the even positive continued fraction expansion of $x$. Then the numerator of $[x]^\sharp_{q,0}$ is the generating function of order ideals $I$ in the fence poset $F(a_1,a_2,\cdots,a_{2n}-1)$ such that $v_0\in I$ iff $v_1\in I$.
    \item[(ii)] Denote by $x = [a_1,a_2,\cdots,a_{2n+1}]$ the odd positive continued fraction expansion of $x$. Then the numerator of $[x]^\flat_{q,0}$ is the generating function of order ideals $I$ in the fence poset $F(0,a_1-1,a_2,\cdots,a_{2n+1}-1)$.
\end{itemize}
\end{proposition}

\begin{proof}
Denote $[x]^\sharp_{q,1} = \frac{A_1}{B_1}$. The numerator of $[x]^\sharp_{q,0}$ is $qA_1+(1-q)B_1$. This polynomial was considered by Morier-Genoud and Ovsienko in Proposition~A.1 and Proposition~A.3 of \cite{MGO-2020}, giving directly the combinatorial interpretation of the numerator of $[x]_{q,0}^\sharp$. \\
\noindent Denote now $[x]_{q,1}^\flat = \frac{A^\flat_1}{B_1^\flat}$. The numerator of $[x]_{q,0}^\flat$ is $(qA_1^\flat + (1-q)B_1^\flat)/(q^2-q+1)$. To obtain a polynomial expression of this numerator, we consider the auxiliary rational number $y = \frac{x}{x-1} = \begin{pmatrix}
1 & 0\\
1 & -1\\
\end{pmatrix}\cdot x$. \\
\noindent The even positive continued fraction expansion of $y$ is $[1,a_1-1,a_2,\cdots,a_{2n+1}]$. By \cite[Theorem 4]{MGO-2020}, the numerator of $[y]_{q,1}^\sharp$ is the generating function of order ideals in the fence poset $F(0,a_1-1,a_2,\cdots,a_{2n+1}-1)$. \\
\noindent Since $(q,0)$-matrices of determinant $-1$ exchange left and right deformations \cite{Jouteur}, we have $[x]_{q,0}^\flat~=~T_qS_{q,0}T_qN_{q,0}\cdot [y]_{q,0}^\sharp~=~T_qS_{q,0}T_qN_{q,0}U_f \cdot [y]_{q,1}^\sharp$, where $N_{q,0}$ is the deformed negation matrix, given in Remark \ref{Rk:conj-neg-matrix}.\\
Denote $[y]_{q,1}^\sharp = \frac{C_1}{D_1}$. 
\begin{align*}
    [x]_{q,0}^\flat &= \begin{pmatrix}
    q^2-q+1 & q-1 \\
    q^2-q+1 & -(q^2-q+1)\\
    \end{pmatrix} U_f \cdot [y]_{q,1}^\sharp\\
    &= \begin{pmatrix}
    q^2-q+1 & q-1 \\
    q^2-q+1 & -(q^2-q+1)\\
    \end{pmatrix} \cdot\frac{qC_1 + (1-q)D_1}{(q^2-q+1)D_1}\\
    &= \frac{C_1}{C_1 - qD_1}.
\end{align*}
Then the numerator of $[x]_{q,0}^\flat$ is $C_1$, which concludes.
\end{proof}

We thank Nicholas Ovenhouse for the following observation.

\begin{proposition}
The numerator of $[x]_{q,0}^\flat$ is the generating function of upper order ideals of the fence poset traditionally associated to $[x]_{q,1}^\sharp$ (see \cite{Ovenhouse}).
\end{proposition}

\begin{proof}
 The poset $F(0,a_1-1,a_2,\cdots,a_{2n+1}-1)$ is the upside down version of the poset $F(a_1-1,a_2,\cdots,a_{2n+1}-1)$, which is the one whose order ideals are counted by the numerator of $[x]_{q,1}^\sharp$.
\end{proof}

\begin{example}
Let us consider $x = \frac{7}{5} = [1,2,1,1] = [1,2,2]$. Its $(q,0)$-deformations are 
\[
\left[\frac{7}{5}\right]_{q,0}^\flat = \frac{1+2q+2q^2+q^3+q^4}{1+q+q^2+q^3+q^4} \text{ and } \left[\frac{7}{5}\right]_{q,0}^\sharp = \frac{1+q+2q^2+q^3+q^4+q^5}{(1-q+q^2)(1+q+2q^2+q^3)}. 
\]
The numerator of $[x]^\sharp_{q,0}$ is the generating function of order ideals in the following fence poset,
\begin{center}
\begin{tikzpicture}[line width=1pt]
\draw[->] (0,1) node {$\bullet$} node[below] {$v_0$} to[bend right] (1,1) node {$\bullet$} node[above] {$v_1$};
\draw[->] (1,1) to[bend right] (0,1);
\draw (1,1) node {$\bullet$} -- (2,0) node {$\bullet$} node[above] {$v_2$};
\draw (2,0) node {$\bullet$} -- (3,-1) node {$\bullet$} node[below] {$v_3$};
\draw (3,-1) node {$\bullet$} -- (4,0) node {$\bullet$} node[above] {$v_4$};
\end{tikzpicture}
\end{center}
\noindent where the two arrows between $v_0$ and $v_1$ means that they must come together in the ideal. \\
The left numerator is the generating function of the fence poset below.
\begin{center}
\begin{tikzpicture}[line width=1pt]
\draw (0,0) node {$\bullet$} node[below] {$v_0$} -- (1,1) node {$\bullet$} node[below] {$v_1$};
\draw (1,1) node {$\bullet$} -- (2,2) node {$\bullet$} node[above] {$v_2$};
\draw (2,2) node {$\bullet$} -- (3,1) node {$\bullet$} node[below] {$v_3$};
\end{tikzpicture}
\end{center}
\end{example}

Rational numbers greater than $1$ parametrize a family of knots called rational knots, see for example \cite{KL_rational_tangles}. It was observed in \cite[Appendix A]{MGO-2020} that the deformation of $x$ (the right $(q,1)$-deformation) gives a way to compute the Jones polynomial of the rational knot parametrized by $x$. With the right $(q,0)$-deformation, the computation is even more direct.

\begin{proposition}  For any $x>1$, the numerator of $[x]_{q,0}^\sharp$ is (up to a power of $q$ and an overall sign) the Jones polynomial of the rational knot associated to $x$ with variable $a=-q^{-1}$.
\end{proposition}

\begin{proof}
 This follows from the formula \eqref{Eq:rel-t-0} expressing $(q,0)$-rationals in terms of $(q,1)$-rationals, which coincides with the formula for the Jones polynomial given in \cite[Proposition A.1]{MGO-2020}.
\end{proof}

\begin{Remark}
Via the transition map defined in \cite{Thomas}, we get that the numerator of $[x]^{\sharp}_{q,0}$ is actually the same as the numerator of $[x]^{\flat}_{q,1}$. And indeed, the numerators of left $(q,1)$-rationals also give the Jones polynomials of rational knots, with variable $a = -q$, see \cite[Appendix A.2]{BBL}.  
\end{Remark}

\begin{example}
The rational knot associated to $\frac{7}{5}$ is the alternating knot with $5$ crossings denoted $5\_ 2$ in the Rolfsen knot table \cite{knot_atlas}, see Figure \ref{fig:5_2}. Its Jones polynomial is $V(a)=-a^{-6} + a^{-5} -a^{-4} + 2a^{-3} -a^{-2} + a^{-1}$. Replacing $a$ by $-q^{-1}$ and multiplying by $-q^{-1}$, we recover the numerator of $\left[\tfrac{7}{5}\right]_{q,0}^\sharp$. 
\end{example}

\begin{figure}[H]
    \centering
    \includegraphics[height=3cm]{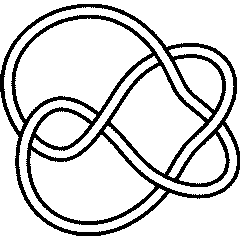}
    \caption{Knot $5\_ 2$ from Rolfsen knot table in the Knot Atlas \cite{knot_atlas}}
    \label{fig:5_2}
\end{figure}

\bibliographystyle{plain}
\bibliography{ref}
%\nocite{*}

\end{document}